\documentclass[12pt]{article}

\usepackage{amsfonts,amsmath,amsthm}
\usepackage[margin=1.5in]{geometry}
\usepackage[english,francais]{babel}
\usepackage[T1]{fontenc}
\usepackage{color}

\newtheorem{thm}{Theorem}
\newtheorem{lem}[thm]{Lemma}
\newtheorem{prop}[thm]{Proposition}
\newtheorem{cor}[thm]{Corollary}
\newtheorem{remark}{Remark}

\DeclareMathOperator{\diam}{diam}
\DeclareMathOperator{\dist}{dist}

\def \NN {\mathbb{N}}
\def \RR {\mathbb{R}}

\def \ra{\rightarrow}

\def \tM{\tilde{M}}
\def \tP{\tilde{P}}
\def \tG{\tilde{G}}

\title{Martin representation and Relative Fatou Theorem   for  fractional Laplacian
with a gradient perturbation}

\author{Piotr Graczyk \\
\scriptsize  LAREMA, Universit\'e d'Angers, 
2 Bd Lavoisier, \\
 \scriptsize 49045 Angers Cedex 1, France\\
 \scriptsize Piotr.Graczyk@univ-angers.fr\\
Tomasz Jakubowski \\
\scriptsize Institute of Mathematics, Wroc\l{}aw University of Technology, Wyb. Wyspia\'nskiego 27,\\
\scriptsize 50-370 Wroc\l{}aw, Poland\\
\scriptsize Tomasz.Jakubowski@pwr.wroc.pl\\ 
Tomasz  Luks\\
\scriptsize LAREMA, Universit\'e d'Angers, 
2 Bd Lavoisier, \\
 \scriptsize 49045 Angers Cedex 1, France\\
\scriptsize luks@math.univ-angers.fr\tiny}
\date{\empty}
\begin{document}
\maketitle
\selectlanguage{english}
\begin{abstract}
Let $L=\Delta^{\alpha/2}+ b\cdot\nabla$ with $\alpha\in(1,2)$.
We prove the Martin representation and the Relative Fatou Theorem for non-negative singular $L$-harmonic functions 
on ${\mathcal C}^{1,1}$ bounded open sets.
\footnote{This research was partially supported by ANR-09-BLAN-0084-01 and grants MNiSW N N201 397137, MNiSW N N201
422539.
}
\end{abstract}
\selectlanguage{english} 

\section{Introduction and Preliminaries}

Analysis of  harmonic functions related to
fractional powers $\Delta^{\alpha/2}$  of the Laplace operator is an important topic,
intensely developed in recent years, also
for perturbations of the operator $\Delta^{\alpha/2}$, see, e.g., \cite{AlibImbert, BrandKarch, FinoKarch, KBTKMK, Caff, TJ2, MR, KBbook} and references therein. 

From the probabilistic point of view,   
stable stochastic processes with gradient perturbations on $\RR^d$, $d\ge 2$, i.e. with the infinitesimal generator
\begin{equation}\label{eq:L}
 L=\Delta^{\alpha/2}+ b\cdot\nabla,
 \end{equation}
where $\alpha\in(0,2)$, constitute  an important class
of jump processes, intensely studied in recent years. Their most  celebrated case are
the Ornstein-Uhlenbeck
stable processes with $b(x)=\lambda x$, $\lambda\in\RR$. They have important physical and financial applications and form a part of
L\'evy-driven Ornstein-Uhlenbeck processes, cf. \cite{Novikov1,Novikov2}.

The motivations of this paper were  to:\\

(i) establish the theory of the  Martin representation
for singular $L$-harmonic non-negative functions

(ii) study  boundary limit properties of  $L$-harmonic functions and
to obtain  a Relative Fatou Theorem for them

(iii) develop the theory of Hardy spaces of $L$-harmonic functions.\\

The topics (i) and (ii) are addressed in this article and the subject (iii) in a forthcoming
paper.

All  these topics are fundamental for the knowledge of $L$-harmonic functions.
 The topics (i) and (ii) are
well developed for  fractional Laplacians.
The Martin representation was established in this case in \cite{KBHirosh, CS1, SW, KMKS, KBTKMK}, see also \cite{KW} 
for a general setting of Markov processes. 
The Relative Fatou Theorem was proved for $\alpha$-harmonic functions on $\mathcal C^{1,1}$ sets in \cite{BD}, 
on Lipschitz sets in \cite{MRFatou}, and on the so-called $\kappa$-fat sets in \cite{Kim}, 
see also the survey \cite[Chapter 3]{KBbook}. 
Furthermore, some important  variants of stable processes such as relativistic, censored 
and truncated stable processes were studied from the point of view of topics (i) and (ii),
see \cite{CK2, CK, KimL, KimS} and  \cite[Section 3.4]{KBbook}. Nevertheless, the methods
of these extensions do not apply to the operator $L$ of the form (\ref{eq:L}).
Let us notice that all our results are also true for Ornstein-Uhlenbeck stable processes.

On the other hand, Martin representations and boundary properties of harmonic functions were 
widely studied in the case of diffusion operators, see, e.g., \cite{Stein1, Stein2, Ancona, Wu, Doob, JerisonKenig, Bass, Aikawa, ArmGard} for  
the results on the classical Laplacian $\Delta=\sum_{i=1}^d\partial^2/\partial x_i^2$, and \cite{Widman, CrFZ, CranstonZ, Murata, Pinchover, TLuksStudia} 
for its various generalizations.

Let us mention that the methods of this article give also interesting new results
for operators different from $L$. 
In the case of  Laplacians with gradient perturbations, i.e. $\alpha=2$, we get new perturbation formulas for the Green function
and the Martin and Poisson kernels (Section  \ref{rem:diffusionCase}).
\\

The potential theory of stable stochastic processes with gradient perturbations
was started in the Ornstein--Uhlenbeck case by T. Jakubowski \cite{TJ1,TJ2}.
Next, in the general context of gradient perturbations and $\alpha>1$,
with a function $b$ from the Kato class ${\mathcal K_d^{\alpha-1}}$
it was developed by K. Bogdan and T. Jakubowski \cite{KBTJ0,KBTJ} and by Z.-Q. Chen, P. Kim and R. Song \cite{CKS}.
Our work is a natural continuation of the research presented in \cite{KBTJ}.

In particular, we send the reader to \cite{KBTJ} for the definitions of
the fractional Laplacian, ${\mathcal C}^{1,1}$ sets,  Green functions and Poisson kernels,
with respect to both operators $\Delta^{\alpha/2}$ and $L$.  The definitions of
$\alpha$-harmonic, regular $\alpha$-harmonic and singular  $\alpha$-harmonic functions
can be found e.g. in the monography \cite[page 61]{KBbook} and are analogous for
$L$-harmonic functions. \\

Throughout this paper, like in \cite{KBTJ}, we suppose $1<\alpha<2$, unless stated otherwise. We consider an open set 
$D$ of class ${\mathcal C}^{1,1}$ and a 
vector field $b\in {\mathcal K_d^{\alpha-1}}$ on $\RR^d$ i.e.
$$
\lim_{\epsilon\ra 0}\sup_{x\in\RR^d} \int_{|x-z|<\epsilon} |b(z)| |x-z|^{\alpha-1-d} dz=0.
$$
The potential theory objects related to the operator $L$ defined in (\ref{eq:L}) will be denoted
with a tilde $\tilde\ $,
while those related to  the operator $\Delta^{\alpha/2}$ will be denoted without it.
In particular $\tilde G_D$ is the Green function of $D$ for $L$
and $G_D$ is the Green function of $D$ for $\Delta^{\alpha/2}$.
We fix throughout this paper a point $x_0\in D$ and define the Martin kernel of $D$ for $\Delta^{\alpha/2}$ 
by
$$
M_D(x,Q)=\lim_{y\ra Q}\frac{G_D(x,y)}{G_D(x_0,y)},\ \ x\in D, Q\in \partial D.
$$
The $L$-Martin kernel is  defined by
$$
\tilde M_D(x,Q)=\lim_{y\ra Q}\frac{\tilde G_D(x,y)}{\tilde G_D(x_0,y)},\ \ x\in D, Q\in \partial D
$$
and we show in Section \ref{sec:Martin} its existence.

The starting point of the research contained in this paper are  the following
mutual estimates of Green functions and Poisson kernels for $L$ and $\Delta^{\alpha/2}$ (see \cite[Theorem 1 and (72)]{KBTJ}).\\[1mm]
{\bf Comparability Theorem.}
{\it
There exists a constant $C=C(\alpha,b,D)$ such that for all $x,y\in D$ and $z \in (\overline D)^c$},
 \begin{equation}\label{eq:GreenCompar}
C^{-1} G_D(x,y)\le \tilde  G_D(x,y) \le C  G_D(x,y),
 \end{equation}
 \begin{equation}\label{eq:PoissCompar}
C^{-1} P_D(x,z) \le \tilde  P_D(x,z)\le C P_D(x,z).
 \end{equation}

 One of the main elements of the proof of (\ref{eq:GreenCompar})
 is the following perturbation formula, that will be also very useful in our present
 work (see \cite[Lemma 12]{KBTJ}).\\[1mm]
{\bf  Perturbation formula for Green functions.}
{\it 
Let $x,y\in\RR^d, x\not=y.$ We have
\begin{equation}\label{eq:pertG}
 \tilde G_D(x,y)=  G_D(x,y)+\int_D \tilde G_D(x,z) b(z)  \cdot \nabla_z  G_D(z,y) dz.
 \end{equation}
}
We start our paper by proving in Section \ref{sec:prepa1}  a generalization of
the Comparability Theorem: according to Lemma \ref{lem:PoissonKer}, the constant $C$
in  the estimates (\ref{eq:GreenCompar}) may be chosen  the same for
sets $D_r$ sufficiently close to $D$. The same phenomenon holds also for
the Poisson kernels $\tilde  P_D(x,y)$ and $  P_D(x,y)$.
In Section   \ref{sec:prepa2} we prove a uniform integrability result, that will be
needed in proving the main results of the paper, contained in Sections \ref{sec:Martin}  and \ref{sec:boundary}.

In Section  \ref{sec:Martin}  we  develop the Martin theory of
$L$-harmonic functions. We prove the existence of the $L$-Martin kernel which is $L$-harmonic
(Theorems \ref{th:MartinPerturb} and \ref{th:HarmonicityMartin}).
 Next we obtain   the Martin representation of singular $L$-harmonic
non-negative functions on $D$, see Theorem \ref{th:MartinRepr}.

The formula (\ref{eq:pertG}) allows us to prove  very useful perturbation formulas for Martin kernels (\ref{eq:PerturbMartin}), Poisson kernels (\ref{eq:pertP})
and singular $\alpha$-harmonic functions (\ref{eq:perturb_v}). 
 In Section \ref{rem:diffusionCase},   (\ref{eq:pertG}) and  (\ref{eq:PerturbMartin})
 are proved 
in the diffusion case $\alpha=2$. Also a  perturbation formula (\ref{eq:pertPoissDiff}) for the $L$-Poisson kernel is derived.

Section \ref{sec:boundary} is devoted to an important fine boundary property
of singular $L$-harmonic functions: the Relative Fatou Theorem (Theorem \ref{th:Fatou}).
We provide a proof of this theorem based on the perturbation formula for singular $\alpha$-harmonic functions (\ref{eq:perturb_v}).

\section{Preparatory results}\label{sec:prepa}

In this section we prove some results, interesting independently,
that will be useful in proving the main results of the paper, coming in the next 
sections.
\subsection{Uniform comparability of Green functions and Poisson kernels}\label{sec:prepa1}
In what follows, $\RR^d$ denotes the Euclidean space of dimension $d\ge
2$, $dy$ stands for the Lebesgue measure on $\RR^d$. Without further mention we will only consider Borelian sets, measures and functions in $\RR^d$.
By $x\cdot y$ we denote the Euclidean scalar product of $x,y\in \RR^d$. Writing $f \approx g$ we mean that there is a constant $C>0$ such that $C^{-1}g \le f \le Cg$.  
As usual, $a \land b =\min(a,b)$ and $a\vee b = \max(a,b)$. 
We let $B(x,r)=\{y\in \RR^d: |x-y|<r\}$.
For $U\subset \RR^d$ we denote 
$$\delta_U(x) =\dist(x,U^c)\,,$$ 
the distance to the complement of $U$. 
\begin{center} \em In what follows $D$ is a bounded $\mathcal C^{1,1}$ open set. \end{center}
For $r\geq0$ define 
$$ D_r=\left\{x\in D: \delta_D(x)>r\right\}. $$ 
When $r$ is sufficiently small, then $D_r$ is also a $\mathcal C^{1,1}$ open set, 
 see \cite[Lemma 5]{MR}, and one may show that the localization radius of $D_r$
 varies continuously with respect to $r$.
 
In the sequel we will often use the estimates of the Green function (\cite{TK}, \cite{CS}, see also \cite{TJ-pms}) of a $\mathcal C^{1,1}$ open set
  \begin{equation}\label{eq:Green}
 G_D(y,z)\approx |y-z|^{\alpha-d} \left( \frac{\delta_D(y)^{\alpha/2}\delta_D(z)^{\alpha/2} }{|y-z|^\alpha } \wedge 1\right),
  \end{equation}
  and of the Martin kernel (\cite{CS1})
  \begin{equation}\label{eq:Martin}
 M_D(x,Q)\approx \frac{\delta_D(x)^{\alpha/2}}{|x-Q|^d }.
  \end{equation}
  Moreover,  in the stable case, the estimates (\ref{eq:Green})  are uniform when we consider the sets $D_r$ sufficiently close to $D$,
  i.e. there exist constants $c,\epsilon_0>0$ depending only on $D$ and $\alpha$ such that for all $r\in [0,\epsilon_0]$ and $x,y\in D_r$ 
 we have
 \begin{eqnarray}\label{eq:GreenUnif}
 c^{-1} |y-z|^{\alpha-d} \left( \frac{\delta_{D_r}(y)^{\alpha/2}\delta_{D_r}(z)^{\alpha/2} }{|y-z|^\alpha } \wedge 1\right) \le
 G_{D_r}(y,z)\\
 \le c|y-z|^{\alpha-d} \left( \frac{\delta_{D_r}(y)^{\alpha/2}\delta_{D_r}(z)^{\alpha/2} }{|y-z|^\alpha } \wedge 1\right),
 \nonumber \end{eqnarray} 
  see \cite[Theorem 21]{TJ-pms} and \cite[Lemma 5]{MR}. We will  now show analogous uniformity of constants for the fractional Laplacian
   with a gradient perturbation.
  
 \begin{lem}\label{lem:PoissonKer}
 (i) There exist constants $c,\epsilon_0>0$ depending only on $D$ and $\alpha$ such that for all $r\in [0,\epsilon_0]$ and $x,y\in D_r$ 
 we have
 $$ c^{-1}G_{D_r}(x,y)\leq \tilde{G}_{D_r}(x,y)\leq cG_{D_r}(x,y).$$

 (ii)There exist constants $C,\epsilon_0>0$ depending only on $D$ and $\alpha$ such that for all $r\in [0,\epsilon_0]$, $x\in D_r$ and $y\in D^c_r$ we have
 $$ C^{-1}P_{D_r}(x,y)\leq \tilde{P}_{D_r}(x,y)\leq CP_{D_r}(x,y).$$
 \end{lem}
 \begin{proof}
 In order to show (i) we follow the proof of the Theorem 1 in \cite{KBTJ}.  We analyse below its crucial points.
 
  1. {\it Comparison of Green functions $\tilde G_S(x,y)$ and $G_S(x,y)$ for "small" sets $S$}, \cite[Lemma 13]{KBTJ}, based
 on estimates from  \cite[Lemma 11]{KBTJ}. Thanks to property (\ref{eq:GreenUnif}), we see that
 the comparison of Green functions $\tilde G_{S_r}(x,y)$ and $G_{S_r}(x,y)$ for small  sets $S$
 holds with a common constant $c$, when $r\in [0,\epsilon_0]$.
 
 2. {\it Harnack inequalities for $L$  and the Boundary  Harnack Principle}, \cite[Lemmas 15, 16]{KBTJ}.
 Thanks to 1.,  we get them uniformly  with respect to $r\in [0,\epsilon_0]$. 
 
 3. Now the proof of (i) for any $\mathcal C^{1,1}$ open set $D$ is the same as in Section 5 of \cite{KBTJ}.
 
 The part (ii) is implied by (i), applying the Ikeda-Watanabe formula for the Poisson kernel $\tilde P_D$, see
 \cite[Lemma 6 and (39)]{KBTJ}. 
Recall that the L\'evy system for the process $\tilde X_t$ is given by the L\'evy measure
of the $\alpha$-stable process $X_t$. 
    \end{proof}
    
\noindent An immediate consequence of Lemma~\ref{lem:PoissonKer} and \cite[Theorem 22]{TJ-pms} is the following uniform estimate of the Poisson kernels 
of $L$ for $D_r$.
\begin{cor}\label{cor:LPoissonEst} 
There exist positive constants $C,\epsilon_0$ depending only on $D$, $\alpha$ and $b$ such that for all 
$r\in[0,\epsilon_0]$, $x\in D_r$ and $y\in D_r^c$ we have
 $$
  \frac{C^{-1}\delta_{D_r}^{\alpha/2}(x)}{\delta_{D_r}^{\alpha/2}(y)(1+\delta_{D_r}(y))^{\alpha/2}|x-y|^d}\leq
 \tilde P_{D_r}(x,y)\leq \frac{C\delta_{D_r}^{\alpha/2}(x)}{\delta_{D_r}^{\alpha/2}(y)(1+\delta_{D_r}(y))^{\alpha/2}|x-y|^d}.
$$
\end{cor}

\subsection{Derivatives of the Poisson kernel   for $\Delta^{\alpha/2}$}\label{sec:prepa2}

In this section we prove   useful gradient estimates for the Poisson kernel of $\Delta^{\alpha/2}$ for $D$, $0<\alpha<2$.

Consider a ball $B=B(\xi_0,r)\subset\overline{B}\subset D$  and let $P_B$ be the Poisson kernel of $B$
\begin{equation}\label{eq:Poissonball}
 P_B(x,y)=C_\alpha^d \left[ \frac{r^2-|x-\xi_0|^2}{|y-\xi_0|^2-r^2}\right]^{\alpha/2}|x-y|^{-d},\ \ x\in B, y\in \bar B^c
\end{equation}
and equal to 0 elsewhere. By \cite[Lemma 3.1]{KBTKAN},
\begin{equation}
	\label{eq:gradPoisson}
|\nabla_x P_B(x,y)| \le (d+\alpha)\frac{P_B(x,y)}{r-|x-\xi_0|}, \qquad x \in B, y \in (\overline{B})^c.
\end{equation}
We will now show analogous estimate for $\mathcal C^{1,1}$ bounded open sets.
\begin{lem}\label{GradEstimPoisson}
Suppose $0<\alpha<2$. Let $D$ be a bounded $\mathcal C^{1,1}$ open set in $\RR^d$ and let $P_D(x,y)$ be the Poisson kernel of $\Delta^{\alpha/2}$ for $D$. Then we have
\begin{equation}\label{eq:6}
|\nabla_xP_D(x,y)| \le (\alpha+d) \frac{P_D(x,y)}{\delta_D(x)}\,, \quad x\in D\,,y\in(\overline{D})^c.
\end{equation}
\end{lem}

\begin{proof}
For $x\in D$ denote $B_x=B(x,\delta_D(x))$. In view of \cite[(29)]{KBTKMK} we have
$$
P_D(x,y)=P_{B_x}(x,y)+\int_{B^c_x}P_{B_x}(x,z)P_D(z,y)dz.
$$
By (\ref{eq:gradPoisson}) and bounded convergence we have 
$$
|\nabla_xP_D(x,y)|\leq |\nabla_xP_{B_x}(x,y)|+|\nabla_x\int_{B^c_x}P_{B_x}(x,z)P_D(z,y)dz|
$$
$$
\leq (d+\alpha)\frac{P_{B_x}(x,y)}{\delta_D(x)}+\int_{B^c_x}|\nabla_xP_{B_x}(x,z)|P_D(z,y)dz
\leq (d+\alpha)\frac{P_D(x,y)}{\delta_D(x)}.
$$
\end{proof}

From (\ref{eq:gradPoisson}) and the dominated convergence theorem it follows, that if $f$ is $\alpha$-harmonic in $D$ then
\begin{equation}
		\label{eq:gradharmonic}
\frac{\partial}{\partial{x_i}}f(x) = \int_{B^c} \frac{\partial}{\partial{x_i}}P_B(x,y) f(y)\, dy, \qquad i=1,\ldots,d.
\end{equation} 
The estimate (\ref{eq:gradPoisson}) and (\ref{eq:gradharmonic}) gives (\cite[Lemma 3.2]{KBTKAN})
\begin{lem}\label{GradEstim}
Let $U$ be an arbitrary open set in $\RR^d$ and let $\alpha\in(0,2)$. For every non-negative
function $u$ on $\RR^d$ which is $\alpha$-harmonic in $U$, we have
\begin{equation}
  \label{eq:5}
|\nabla u(x)| \le d \frac{u(x)}{\delta_U(x)}\,, \quad x\in U\,.
\end{equation}
\end{lem}
Since $G_U(\cdot,y)$ is $\alpha$-harmonic in $U \setminus\{y\}$, for
every $y \in U$ we obtain
\begin{equation}\label{eq:GradGreen}
|\nabla_x G_U(x,y)| \le d \frac{G_U(x,y)}{\delta_U(x) \land |x-y|}\,,
\quad x,y \in U,\,\; x\neq y\,.
\end{equation}

\subsection{A uniform integrability result}\label{sec:prepa3}

One of the important results of \cite{KBTJ} is 
 \begin{lem}\label{lem:UnifInt}
$G_D(y,w)/[\delta(w)\land |y-w|]$ is uniformly in $y$ 
  integrable against $|b(w)| dw$.
\end{lem}
\noindent In the next lemma we will show a similar property for the family of functions $G_{D_{2^{-n}}}(x,w) M_D(w,Q)\delta_{D_{2^{-n}}}(w)^{-1}$.
\begin{lem}\label{lem:uniforminteg}
Let $x\in D$ be fixed. There exists $N=N(D,x)\in\NN$ such that the functions
$$
G_{D_{2^{-n}}}(x,w)M_D(w,Q)\delta_{D_{2^{-n}}}(w)^{-1}
$$
are uniformly in $Q\in\partial D$ and $n>N$ integrable against $|b(w)|dw$.
\end{lem}

\begin{proof}
In view of the properties of $D$ and of the estimates of $G_{D_{2^{-n}}}(x,w)$ and $M_D(w,Q)$, we can choose $N=N(D,x)\in\NN$ sufficiently large, such that for all $n>N$ we have
$$
\frac{G_{D_{2^{-n}}}(x,w)M_D(w,Q)}{\delta_{D_{2^{-n}}}(w)}\leq 
\frac{c{\bf 1}_{D_{2^{-n}}}(w)\delta_{D}(w)^{\alpha/2}}{\delta_{D_{2^{-n}}}(w)^{1-\alpha/2}|w-Q|^d|w-x|^{d-\alpha}}
$$
$$
\leq \tilde{c}{\bf 1}_{D_{2^{-n}}}(w)\left(|w-x|^{\alpha-d}+\frac{\delta_{D}(w)^{\alpha/2}}{\delta_{D_{2^{-n}}}(w)^{1-\alpha/2}|w-Q|^d}\right),
$$
where $c$ and $\tilde{c}$ depends only on $D,\alpha$ and $x$. The first term in the parentheses is integrable against $|b(w)|dw$ independently of $Q,n$, so we only 
need to consider the second one. For $w\in D_{2^{-N}}$ and $Q\in\partial D$ we have
$$
\frac{\delta_{D}(w)^{\alpha/2}}{\delta_{D_{2^{-n}}}(w)^{1-\alpha/2}|w-Q|^d}\leq \diam(D)^{\alpha/2}2^{(N+1)(d+1-\alpha/2)},
$$
and for $w\in D_{2^{-n}}\setminus D_{2^{-N}}$, $Q\in\partial D$ we get
$$
\frac{\delta_{D}(w)^{\alpha/2}}{\delta_{D_{2^{-n}}}(w)^{1-\alpha/2}|w-Q|^d}=
\frac{(\delta_{D_{2^{-n}}}(w)+2^{-n})^{\alpha/2}}{\delta_{D_{2^{-n}}}(w)^{1-\alpha/2}|w-Q|^d}
$$
$$
\leq 2^{\alpha/2}\left(|w-Q|^{\alpha-d-1}+\frac{2^{-n\alpha/2}}{\delta_{D_{2^{-n}}}(w)^{1-\alpha/2}|w-Q|^d}\right).
$$
Since ${\bf 1}_{D_{2^{-n}}}(w)|w-Q|^{\alpha-d-1}$ is uniformly in $Q,n$ integrable against $|b(w)|dw$, 
we can restrict our attention to the function
$$
H_n(w,Q)=\frac{2^{-n\alpha/2}{\bf 1}_{D_{2^{-n}}}(w)}{\delta_{D_{2^{-n}}}(w)^{1-\alpha/2}|w-Q|^d}
$$
Let $R>N$, $R\in\NN$. For $k,m,n\in\NN$, $k\geq R$, $m\geq N$, $n>N$ we define 
$$
W^n_{k,m}(Q,R)=\left\{w\in D_{2^{-n}}: \frac{1}{2^{k+1}}<\delta_{D_{2^{-n}}}(w)\leq \frac{1}{2^k},\quad \frac{1}{2^{m+1}}<|w-Q|\leq \frac{1}{2^{m}}\right\},
$$
$$
W^n_k(Q,R)=\left\{w\in D_{2^{-n}}: \frac{1}{2^{k+1}}<\delta_{D_{2^{-n}}}(w)\leq \frac{1}{2^k},\quad |w-Q|>\frac{1}{2^N}\right\}.
$$
We note that $W^n_{k,m}(Q,R)=\emptyset$ for $k<m$ or $m\geq n$. $W^n_{k,m}(Q,R)$ can be covered by $c_1(2^{k-m})^{d-1}$ balls of radii $2^{-k}$, where 
$c_1=c_1(D)$. For $r>0$ denote
$$
K_r=\sup_{z\in\RR^d}\int_{B(z,r)}|b(w)||z-w|^{\alpha-d-1}dw.
$$
Then $K_r\to0$ as $r\downarrow0$. We have
$$
\int_{W^n_{k,m}(Q,R)}H_n(w,Q)|b(w)|dw
$$
$$
\leq (2^{k+1})^{1-\alpha/2}(2^{m+1})^d2^{-n\alpha/2}c_1(2^{k-m})^{d-1}\sup_{z\in D}\int_{B(z,2^{-k})}|b(w)|dw
$$
$$
\leq (2^{k+1})^{1-\alpha/2}(2^{m+1})^d2^{-n\alpha/2}c_1(2^{k-m})^{d-1}(2^k)^{\alpha-d-1}K_{2^{-k}}
$$
$$
\leq c_2K_{2^{-R}}(2^{k})^{\alpha/2-1}2^{m}2^{-n\alpha/2},
$$
where $c_2=c_2(D,b,\alpha)$. Furthermore, $W^n_{k}(Q,R)$ can be covered by $c_3(2^{k})^{d-1}$ balls of radii $2^{-k}$, where $c_3=c_3(D)$, and thus
$$
\int_{W^n_{k}(Q,R)}H_n(w,Q)|b(w)|dw
$$
$$
\leq (2^{k+1})^{1-\alpha/2}2^{Nd}2^{-n\alpha/2}c_3(2^{k})^{d-1}\sup_{z\in D}\int_{B(z,2^{-k})}|b(w)|dw
$$
$$
\leq (2^{k+1})^{1-\alpha/2}2^{Nd}2^{-n\alpha/2}c_3(2^{k})^{d-1}(2^k)^{\alpha-d-1}K_{2^{-k}}
$$
$$
\leq c_4K_{2^{-R}}(2^{k})^{\alpha/2-1}2^{-n\alpha/2},
$$
where $c_4=c_4(D,b,\alpha)$. Let $A^n_R=\left\{w\in D_{2^{-n}}: \delta_{D_{2^{-n}}}(w)\leq 2^{-R}\right\}$. Then 
$$
A^n_R=\sum^{\infty}_{k=R}W^n_{k}(Q,R)+\sum^{n-1}_{m=N}\sum^{\infty}_{k=R\vee m}W^n_{k,m}(Q,R),
$$
and we obtain
$$
\int_{A^n_R}H_n(w,Q)|b(w)|dw
$$
$$
\leq c_4K_{2^{-R}}2^{-n\alpha/2}\sum^{\infty}_{k=R}(2^{k})^{\alpha/2-1}+
\sum^{n-1}_{m=N}\sum^{\infty}_{k=R\vee m}c_2K_{2^{-R}}(2^{k})^{\alpha/2-1}2^{m}2^{-n\alpha/2}
$$
$$
\leq c_5K_{2^{-R}}\left(2^{-n\alpha/2}(2^{R})^{\alpha/2-1}+\sum^{n-1}_{m=N}(2^{n-m})^{-\alpha/2}\right)\leq c_6K_{2^{-R}},
$$
where $c_6=c_6(D,b,\alpha)$. For $w\in D_{2^{-n}}\setminus A^n_R$ we have
$$
H_n(w,Q)<2^{-n\alpha/2}(2^R)^{1-\alpha/2}\left(\frac{1}{2^R}+\frac{1}{2^n}\right)^{-d}<4^{dR},
$$
so $B^n_R(Q):=\left\{w:H_n(w,Q)> 4^{dR}\right\}\subset A^n_R$ for all $Q\in\partial D$ and $n>N$. Therefore
$$
\lim_{R\to\infty}\sup_{Q\in\partial D,n>N}\int_{B^n_R(Q)}H_n(w,Q)|b(w)|dw\leq\lim_{R\to\infty}c_6K_{2^{-R}}=0.
$$
\end{proof}

 \section{Martin kernel and Martin representation}\label{sec:Martin}
 
 In this section we will discuss first the existence and the properties of the
Martin kernel of $L$ for a $\mathcal C^{1,1}$ bounded open set $D$. Next we will investigate the Martin representation 
for non-negative singular $L$-harmonic functions on $D$.

\subsection{Existence and Perturbation formula for the  $L$-Martin kernel}

In order to prove the existence of the $L$-Martin kernel, we will need the
following property of the Green function for $\Delta^{\alpha/2}$.

\begin{lem}\label{lem:ratioGreen}
For all $x\in D$ and $Q\in\partial D$ we have
$$
\lim_{y\to Q}\frac{\nabla_xG_D(x,y)}{G_D(x_0,y)}=\nabla_x M_D(x,Q).
$$
\end{lem}

\begin{proof}
Let $z\in D,Q\in\partial D$ and choose $r>0$ such that $\overline{B(z,r)}\subset D$ and $B(z,r)\cap B(Q,r)=\emptyset$. Since $G_D(\cdot,y)$ is $\alpha$-harmonic in $B(z,r)$ for $y\in B(Q,r)\cap D$, by (\ref{eq:gradharmonic}), we have
\begin{align*}
\frac{\nabla_xG_D(x,y)}{G_D(x_0,y)}& = \nabla_x\int_{B(z,r)^c}P_{B(z,r)}(x,w)\frac{G_D(w,y)}{G_D(x_0,y)}dw \\
&=\int_{B(z,r)^c}\nabla_x P_{B(z,r)}(x,w)\frac{G_D(w,y)}{G_D(x_0,y)}dw, \qquad x \in B(z,r).
\end{align*}
Furthermore, by (\ref{eq:gradPoisson}) and (\ref{eq:Green}),
$$
|\nabla_x P_{B(z,r)}(x,w)|\frac{G_D(w,y)}{G_D(x_0,y)}\leq C\frac{P_{B(z,r)}(x,w)}{r-|x-z|}\frac{G_D(w,y)}{\delta_D(y)^{\alpha/2}}.
$$
We now use the estimate \cite[(25)]{KBTJ} and by considering the cases $\delta_D(w)>|w-y|$ and $\delta_D(w)\leq|w-y|$
we get $\frac{G_D(w,y)}{\delta_D(y)^{\alpha/2}}\leq C |w-y|^{\alpha/2-d}$. Hence 
the last term is uniformly in $y\in B(Q,r/2)\cap D$ integrable against $dw$, and thus
$$
\lim_{y\to Q}\int_{B(z,r)^c}\nabla_x P_{B(z,r)}(x,w)\frac{G_D(w,y)}{G_D(x_0,y)}dw
$$
$$
=\int_{B(z,r)^c}\nabla_x P_{B(z,r)}(x,w)M_D(w,Q)dw=\nabla_xM_D(x,Q).
$$
The last equality follows from (\ref{eq:gradharmonic}) and the $\alpha$-harmonicity of the Martin kernel.
\end{proof}

\noindent Thanks Lemma~\ref{lem:ratioGreen} we obtain the main result of this subsection.

 \begin{thm}\label{th:MartinPerturb}
 Let $x\in D$ and $Q\in\partial D$.
   Let $M_D(x,Q)$ be the Martin kernel of $\Delta^{\alpha/2}$ for $D$.
   Denote  
   $$
   l_D(x,Q)=M_D(x,Q)+\int_D \tilde G_D(x,z) b(z)\cdot \nabla_z M_D(z,Q) dz
   $$
The function $l_D(x,Q)$ is well defined for  $x\in D$ and $Q\in\partial D$ and $l_D(x,Q)>0$.
Moreover the following limit exists and equals:
$$
\lim_{y\ra Q} \frac{\tilde G_D(x,y)}{\tilde G_D(x_0,y)}=\frac{l_D(x,Q)}{l_D(x_0,Q)}.
$$
Thus the Martin kernel   of $ L=\Delta^{\alpha/2} +b\cdot \nabla$  for $D$  exists and equals
\begin{equation}\label{eq:PerturbMartin}
\tilde M_D(x,Q) =\frac{1}{l_D(x_0,Q)}\left[ M_D(x,Q)+\int_D \tilde G_D(x,z) b(z)\cdot \nabla_z M_D(z,Q) dz\right].
\end{equation}
 \end{thm}
\begin{proof}
 We divide the perturbation formula (\ref{eq:pertG}) for the Green function $\tilde G_D(x,y)$
 by $G_D(x_0,y)$ and  let $y\ra Q$.
 
 The  exchange of $\lim_{y\ra Q}$ and $\int_D$ is justified by
 Lemma 11 of \cite{KBTJ}, see the formula (49) in its proof.  Note that
 by the Boundary Harnack Principle, $G_D(x_0,y)\approx G_D(x,y)$ when $y\in B(Q,\epsilon_0)$,
 a sufficiently small ball around $Q$.
 We also use the estimates  (\ref{eq:Green}), (\ref{eq:GradGreen}) and   (\ref{eq:GreenCompar}).
 
  The  exchange of $\lim_{y\ra Q}$ and  $\nabla_z$   is justified by Lemma \ref{lem:ratioGreen}.
 Finally 
 $$
 \lim_{y\ra Q} \frac {\tilde G_D(x,y)}{  G_D(x_0,y)} = M_D(x,Q)+\int_D \tilde G_D(x,z) b(z)\cdot \nabla M_D(z,Q) =l_D(x,Q).
 $$
 The strict positivity of the function $l_D(x,Q)$ follows from (\ref{eq:GreenCompar}), which implies that there exists $ a>0$ such that
 \begin{equation}\label{eq:StrictPos_l}
   l_D(x,Q)  \ge a M_D(x,Q)>0.
 \end{equation}

 Now we consider the quotient
 $$
 \frac{\tilde G_D(x,y)}{\tilde G_D(x_0,y)}= \frac{\tilde G_D(x,y)}{  G_D(x_0,y)} \frac{ G_D(x_0,y)}{\tilde G_D(x_0,y)}
 \ra \frac{l_D(x,Q)}{l_D(x_0,Q)},
 $$
 when $y\ra Q$.
\end{proof}

Directly from the definition of $\tilde M_D(x,Q)$ and (\ref{eq:GreenCompar}) we obtain the following corollary.
\begin{cor}\label{cor:MartinComp}
There is a constant $c$ such that for all $x \in D$ and $Q \in \partial D$,
\begin{equation}\label{eq:MartinComp}
c^{-1}M_{D}(x,Q)\leq \tilde{M}_{D}(x,Q)\leq cM_{D}(x,Q).
\end{equation}
 
\end{cor}

\subsection{Properties of the $L$-Martin kernel}

We will now study further properties of the Martin kernel of $L$ for $D$. We start with the following useful formulas.

\begin{lem}\label{lem:PoissonMartin}
 Consider a ${\mathcal C}^{1,1}$ open set $U\subset\overline{U}\subset D$.\\
 (i)\;({\bf Perturbation formula for the Poisson kernel}) For all $x \in U, z \in (\overline U)^c$
 \begin{equation}\label{eq:pertP}
\tP_U(x,z) = P_U(x,z) + \int_U \tG_U(x,w) b(w)\cdot \nabla  P_U(w,z) dw. 
\end{equation}
(ii) Let $Q \in \partial D$. We have the following expression for the $L$-Poisson integral of the Martin kernel of $\Delta^{\alpha/2}$:
  \begin{align}
  \tilde P_U M_D(x,Q) &:=\int_{U^c} \tilde P_U(x,y) M_D(y,Q)dy \nonumber\\
  &=M_D(x,Q)+\int_U \tilde G_U(x,z) b(z)\cdot \nabla M_D(z,Q) dz,\ \ \ x\in U. \label{eq:TildePM}
   \end{align}
\end{lem}
\begin{proof}
In the following we apply the Ikeda-Watanabe formula for the Poisson kernels $\tilde P_U$ and $ P_U$.
  By (\ref{eq:pertG}) and Fubini's theorem, for any $x \in U$ and $z \in U^c$,
\begin{align*}
\tP_U(x,z) &= \int_U \mathcal{A}_{d,\alpha}\frac{\tG_U(x,y)}{|z-y|^{d+\alpha}}dy  \notag\\
& = \int_U \frac{\mathcal{A}_{d,\alpha}}{|z-y|^{d+\alpha}} \left[G_U(x,y) + \int_U \tG_U(x,w) b(w)\cdot \nabla  G_U(w,y) dw\right]dy \notag\\
& = P_U(x,z) + \int_U \tG_U(x,w) b(w)\cdot \nabla  P_U(w,z) dw. 
\end{align*}
For the necessary exchanges of order of integration and derivation in the last formula, we apply (\ref{eq:GradGreen}), (\ref{eq:GreenCompar}), 
Lemma~\ref{lem:UnifInt} and bounded convergence theorem.
In order to prove (ii), we  use (i) and insert the formula (\ref{eq:pertP}) in $\int_{U^c} \tilde P_U(x,y) M_D(y,Q)dy$. 
  We obtain
 \begin{eqnarray*}
 \int_{U^c} \tilde P_U(x,y) M_D(y,Q)dy
= M_D(x,Q)+\int_U \tilde G_U(x,z) b(z)\cdot \nabla M_D(z,Q) dz
\end{eqnarray*}
 In the last equality the use of Fubini theorem and the exchange of $\int$ and $\nabla$  are justified
 by (\ref{eq:6}), (\ref{eq:Martin}), Lemma \ref{lem:UnifInt} and bounded convergence. 
\end{proof}

\begin{lem}\label{lem:continuity}
The Martin kernel $\tilde M_D(\cdot,\cdot)$ is jointly continuous on $D\times\partial D$.
\end{lem}

\begin{proof}
By Theorem~\ref{th:MartinPerturb} and the continuity of $M_D(\cdot,\cdot)$, it suffices to show the joint continuity on $D\times\partial D$ of the function
$$
f(x,Q)=\int_D \tG_D(x,z) b(z)\cdot \nabla M_D(z,Q) dz.
$$
Let $z\in D$. By (\ref{eq:gradharmonic}) and the $\alpha$-harmonicity of $M_D(\cdot,Q)$, for $r>0$ sufficiently small, we have
$$
\nabla M_D(z,Q)=\nabla\int_{B(z,r)^c}P_{B(z,r)}(z,w)M_D(w,Q)dy
$$
$$
=\int_{B(z,r)^c}\nabla P_{B(z,r)}(z,w)M_D(w,Q)dw.
$$
From (\ref{eq:gradPoisson}) and (\ref{eq:Martin}) it follows, that $\nabla P_{B(z,r)}(z,w)M_D(w,Q)$ is uniformly in $Q$ integrable 
against $dw$. This implies that $\nabla M_D(z,\cdot)$ is continuous on $\partial D$ for every $z\in D$. Let now $x\in D$ and choose $r>0$ such that 
$\overline{B(x,r)}\subset D$. By (\ref{eq:GreenCompar}), (\ref{eq:Green}), (\ref{eq:5}) and (\ref{eq:Martin}), for all $y\in B(x,r)$, $z\in D$ and 
$Q\in\partial D$, we have
\begin{equation}\label{eq:GreenMartinEst}
\tG_D(y,z)|\nabla M_D(z,Q)|\leq\frac{C\delta_D(z)^{\alpha-1}}{|y-z|^{d-\alpha}|z-Q|^{d}}
\leq \frac{C}{|y-z|^{d-\alpha}|z-Q|^{d+1-\alpha}}.
\end{equation}
Hence, $\tG_D(y,z)|\nabla M_D(z,Q)|$ is uniformly in $y\in B(x,r)$ and $Q\in\partial D$ integrable against $|b(z)|dz$, which gives the continuity 
of $f(\cdot,\cdot)$.
\end{proof}

\noindent 
We will now use Lemma~\ref{lem:PoissonMartin} to show $L$-harmonicity of $\tilde M(x,Q)$.

 \begin{thm}\label{th:HarmonicityMartin}
 For every $Q\in\partial D$ the Martin kernel $\tilde M(x,Q)$ is a singular $L$-harmonic function of $x$ on $D$. 
 \end{thm}
 \begin{proof}
First consider a ${\mathcal C}^{1,1}$ open set $U=D_r$.
We note that 
\begin{equation}\label{eq:Greenharm}
\tG_D(x,w) = \tG_U(x,w) + \int_{U^c} \tP_U(x,z)\tG_D(z,w)dz.
\end{equation}
By (\ref{eq:TildePM}), (\ref{eq:pertP}), (\ref{eq:Greenharm}) and Fubini's theorem
\begin{align*}
\tP_U l_D(x,Q) &= \int_{U^c} \tP_U(x,z) l_D(z,Q)dz \\
&= \int_{U^c} \tP_U(x,z) M_D(z,Q)dz \\
&+ \int_{U^c} \tP_U(x,z) \int_D \tG_D(z,w) b(w) \cdot \nabla M_D(w,Q) dw dz\\
&=  M_D(x,Q) + \int_{U^c} \int_U \tG_U(x,w) b(w) \cdot \nabla P_U(w,z) dw M_D(z,Q)dz \\
&+ \int_D \left[\int_{U^c} \tP_U(x,z)  \tG_D(z,w) dz\right] b(w) \cdot \nabla M_D(w,Q) dw \\
&=  M_D(x,Q) + \int_U \tG_U(x,w) b(w)\cdot \nabla  M_D(w,Q) dw  \\
&+ \int_D [\tG_D(x,w) - \tG_U(x,w)] b(w) \cdot \nabla M_D(w,Q) dw \\
&=  M_D(x,Q) + \int_D \tG_D(x,w) b(w)\cdot \nabla  M_D(w,Q) dw = l_D(x,Q).
\end{align*}
Thus the function $l_D(x,Q)$ is regular $L$-harmonic on each set $U=D_r$ for $r$ sufficiently small. By the strong Markov property, 
it has the mean value property on each open set $U\subset\bar U\subset D$.
\end{proof}

\subsection{ $ L$-Martin  representation}

The objective of this section is to prove the following  Martin representation theorem for non-negative singular $L$-harmonic functions 
on $D$.

\begin{thm}\label{th:MartinRepr}
 For every non-negative finite measure $\nu$ on $\partial D$ the function $u$ given by
 \begin{equation}\label{eq:MartinIntegral}
 u(x)=\int_{\partial D} \tilde M_D(x,Q)d\nu(Q),
 \end{equation}
 is singular $L$-harmonic on $D$. Conversely, if $u$ is non-negative singular $L$-harmonic on $D$, then there exists a unique non-negative 
 finite measure $\nu$ on $\partial D$ verifying (\ref{eq:MartinIntegral}).
\end{thm}
\begin{proof}
The $L$-harmonicity of the Martin integral (\ref{eq:MartinIntegral}) and the uniqueness of the representation follow from Theorem~\ref{th:HarmonicityMartin},  Lemma~\ref{lem:continuity}, (\ref{eq:Martin}), (\ref{eq:MartinComp}) and Fubini theorem, in the same way as in the case of the Martin representation for $\alpha$-harmonic functions in \cite[proof of Theorem 1]{KBHirosh}. We will now focus on the 
existence part. By $L$-harmonicity of $u$ and by (\ref{eq:pertP}) we have for each $n$
 \begin{eqnarray*}
  u(x)&=&\int_{D_{1/n}^c} \tP_{D_{1/n}}(x,y) u(y) dy=\\
  &=&\int_{D_{1/n}^c} u(y)\left[P_{D_{1/n}}(x,y) +\int_{D_{1/n}}  \tG_{D_{1/n}}(x,w) b(w)\cdot \nabla P_{D_{1/n}}(w,y)dw\right]dy.
 \end{eqnarray*}
 Denote
 $$
 u^*_n(x)= \int_{D_{1/n}^c} P_{D_{1/n}}(x,y) u(y) dy.
 $$
By (\ref{eq:6}), \cite[(72)]{KBTJ} and Lemma \ref{lem:UnifInt} we have
$$
\int_{D_{1/n}^c}\int_{D_{1/n}}\tG_{D_{1/n}}(x,w)|b(w)||\nabla P_{D_{1/n}}(w,y)|u(y)dwdy
$$
$$
\leq C\int_{D_{1/n}}\tG_{D_{1/n}}(x,w)|b(w)|\frac{u(w)}{\delta_{D_{1/n}}(w)}dw<\infty,
$$
where $C=C(\alpha,b,D_{1/n})>0$. Hence, by Fubini theorem
 \begin{eqnarray*}
  u(x)&=&  u^*_n(x) +\int_{D_{1/n}}\tG_{D_{1/n}}(x,w) b(w)\cdot \int_{D_{1/n}^c}\nabla P_{D_{1/n}}(w,y)u(y)dydw.
 \end{eqnarray*}
The function $u^*_n$ is $\alpha$-harmonic on $D_{1/n}$,  so it is differentiable.
In order to justify the exchange of $\int$ and $\nabla$ in the last integral we fix $w\in D_{1/n}$.
Then by (\ref{eq:6}) and \cite[(72)]{KBTJ}, for $\varepsilon>0$ sufficiently small and 
all $w'\in B(w,\varepsilon)$ and $y\in D_{1/n}^c$ we have
$$
|\nabla_{w'} P_{D_{1/n}}(w',y)u(y)|\leq C\frac{u(y)}{\delta_{D_{1/n}}(y)^{\alpha/2}},
$$
where $C=C(\alpha,b,D_{1/n},\varepsilon)>0$. Since the last term is integrable on $D_{1/n}^c$, by the dominated convergence
we obtain
 \begin{eqnarray}\label{eq:unStar}
  u(x)= u^*_n(x) + \int_D  \tG_{D_{1/n}}(x,w) b(w)\cdot \nabla u^*_n(w) dw.
 \end{eqnarray}
We now study the sequence $ u^*_n(x)$ in the same way as K. Bogdan \cite{KBHirosh}
in the proof of the existence part of the $\Delta^{\alpha/2}-$Martin representation, with  the difference that  in our
case the function $u$ under the integral defining $ u^*_n$ is not $\alpha$-harmonic. 

Like in \cite[(2.27)]{KBHirosh}  we have
$$
u^*_n(x)=\int_{D_{1/n}^c} P_{D_{1/n}}(x,y) u(y) dy=\int_{D_{1/n}^c}\int_{D_{1/n}} u(y) \mathcal{A}_{d,\alpha}\frac{G_{D_{1/n} }(x,\xi)}{|\xi-y|^{d+\alpha}}d\xi dy
$$
Set $\mu_n(d\xi)=   \mathcal{A}_{d,\alpha} G_{D_{1/n} }(x_0,\xi)  \int_{D_{1/n}}\frac{ u(y)}{|\xi-y|^{d+\alpha}}  dy d\xi$.
 Lemma \ref{lem:PoissonKer} implies that 
$$
\mu_n(\RR^d)= \int_{D_{1/n}^c} P_{D_{1/n}}(x_0,y) u(y) dy \le C \int_{D_{1/n}^c} \tP_{D_{1/n}}(x_0,y) u(y) dy  =cu(x_0) <\infty
$$
(recall that if $u$ was $\alpha$-harmonic, then $\mu_n(\RR^d)=u(x_0))$.
We obtain
$$
u^*_n(x)=\int_{D_{1/n}}\frac{ G_{D_{1/n} }(x,\xi)}{ G_{D_{1/n} }(x_0,\xi)} \mu_n(d\xi).
$$
The only other property of the function $u$ intervening  in the proof of the existence part of the $\Delta^{\alpha/2}-$Martin representation
in \cite{KBHirosh} is
$$
 \lim_n \int_{D_{1/n}^c} u(y)dy =0
$$
and it also holds in our case: the $L$-harmonic function $u$ is   integrable on $D_{1/n}^c$ for every $n$. 
The sequence  $(\mu_n)$ of simultaneously bounded finite measures with  support contained in $\bar D$
is tight. We choose a subsequence $\mu_{n_k}$ converging to a finite (perhaps zero)  measure $\mu$. This choice
is common for all $x$. Without loss of generality, we may suppose that $(n_k)$ is a subsequence of $(2^{-n})$. 
The limit measure $\mu$ satisfies 
$${\rm supp}(\mu)\subset \partial D.$$
Exactly as in the proof of the existence part of the $\Delta^{\alpha/2}-$Martin representation
in \cite{KBHirosh}, we deduce that 
for all $x\in D$ the limit
$$\lim_{k} u_{n_k}^*(x)=u^*(x) $$
exists and
\begin{eqnarray}\label{eq:uStarMartin}
u^*(x) = \int_{\partial D} M_D(x,Q) d\mu(Q). 
 \end{eqnarray} 
Furthermore, in view of (\ref{eq:gradharmonic}), for $x\in D_{1/n}$ and $r>0$ sufficiently small we have
$$
\nabla u^*_n(x)=\nabla\int_{B(x,r)^c}P_{B(x,r)}(x,y)u^*_n(y)dy=\int_{B(x,r)^c}\nabla P_{B(x,r)}(x,y)u^*_n(y)dy.
$$
By Lemma~\ref{lem:PoissonKer} and (\ref{eq:gradPoisson}) we have
$$
|\nabla P_{B(x,r)}(x,y)u^*_n(y)|\leq C\frac{P_{B(x,r)}(x,y)}{r-|x|}u(y),
$$
and by the dominated convergence we get $\nabla u^*_{n_k}(x)\to\nabla u^*(x)$ as $k\to\infty$.
We also have $\tG_{D_{1/n}}(x,w)\nearrow\tG_D(x,w)$. In order to justify 
the passage with the limit under the integral sign in (\ref{eq:unStar}) with $n_k$ instead of $n$ we observe that
the functions $\tG_{D_{1/{n_k}}}(x,w) b(w)\cdot \nabla u^*_{n_k}(w)$ are uniformly integrable on $D$. Clearly, by Lemma~\ref{lem:PoissonKer} we have 
$c^{-1} u^*_n(w)\leq u(w)\leq c u^*_n(w)$, where $c$ does not depend on $n$, thus $u^*_n(w)\leq c u^*(w)$. 
By the gradient estimates we get
$$
\tG_{D_{1/n}}(x,w)|b(w)||\nabla u^*_n(w)|\leq \tG_{D_{1/n}}(x,w)|b(w)|\frac{u^*(w)}{\delta_{D_{1/n}}(w)},
$$
and the uniform integrability follows from (\ref{eq:uStarMartin}), Lemma~\ref{lem:PoissonKer} and Lemma~\ref{lem:uniforminteg}. 
Therefore 
\begin{eqnarray}\label{eq:uStar}
  u(x)= u^*(x) + \int_{D}  \tG_{D}(x,w) b(w)\cdot \nabla u^*(w) dw,
 \end{eqnarray}
 which, using (\ref{eq:uStarMartin}), becomes
 \begin{eqnarray}\label{eq:uM}
  &&u(x)=\\
  &&= \int_{\partial D} M_D(x,Q) d\mu(Q) + \int_{D}  \tG_{D}(x,w) b(w)\cdot \nabla\int_{\partial D} M_D(w,Q) d\mu(Q)dw.\nonumber
 \end{eqnarray}
 By the gradient estimates and dominated convergence we also get 
 $$
 \nabla\int_{\partial D} M_D(w,Q) d\mu(Q)=\int_{\partial D}\nabla M_D(w,Q) d\mu(Q),\quad w\in D.
 $$ 
 Define a  measure $\nu$ on $\partial D$ by $\nu(dQ)=l_D(x_0,Q) d\mu(Q)$. As the function $Q\ra l_D(x_0,Q)$ is continuous positive,
 the measure  $\nu$ is finite positive on  $\partial D$. Using Fubini theorem in (\ref{eq:uM}) and the perturbation formula for $\tM_D$ from
 Theorem~\ref{th:MartinPerturb}, we obtain 
  \begin{eqnarray*}
  u(x)= \int_{\partial D} \tM_D(x,Q)d\nu(Q). 
 \end{eqnarray*} 
 \end{proof}
 
 \begin{remark}
 We point out that the proof of Theorem~\ref{th:MartinRepr} is based on the perturbation formula. In fact, the methods used in \cite{KBHirosh} 
 in order to prove the Martin representation theorem for singular $\alpha$-harmonic functions can not be applied in the present case because 
 the Green function $\tG_D(x,y)$ is not $L$-harmonic on $D\setminus\left\{x\right\}$ as a function of $y$.
 \end{remark}

 \begin{cor}\label{cor:HarmPerturb}
  ({\bf Perturbation formula for singular $L$-harmonic functions})
 Let $v(x)\ge 0$ be a singular $L$-harmonic function on $D$ with the Martin representation 
  \begin{equation}\label{eq:v} 
v(x)=\int_{\partial D} \tilde M_D(x,Q)d\nu(Q),\ \ \ x\in D.  
\end{equation}
Define
 a singular $\alpha$-harmonic function $v^*$  on $D$ by
  \begin{equation}\label{eq:v_star}
v^*(x)=\int_{\partial D}   M_D(x,Q)\displaystyle{\frac{d\nu(Q)}{l(x_0,Q)}},\ \  x\in D
  \end{equation} 
 Then the following formula holds
 \begin{equation}\label{eq:perturb_v}
v(x)= v^*(x)  + \int_{D}  \tG_{D}(x,w) b(w)\cdot \nabla v^*(w) dw.
  \end{equation}
 \end{cor}
 \begin{proof}
 Observe that by (\ref{eq:StrictPos_l})  there exists $\delta>0$ such that
 $$
 l_D(x_0,Q)>\delta>0
 $$
 for all $Q\in\partial D$. Thus the measure $d\mu(Q)=\frac{d\nu(Q)}{l_D(x_0,Q)}$
 is finite and the function $v^*$ is well defined.
 By the unicity of the Martin representation 
 and  the formula (\ref{eq:uStarMartin}), the function $v^*$ defined by (\ref{eq:v_star}) is the same as
 the function $v^*$ defined by a limit procedure and associated to $v$   in the proof of the Theorem \ref{th:MartinRepr}.
 Hence the formula   (\ref{eq:uStar}) holds for $v$ and $v^*$.
 It is equivalent to (\ref{eq:perturb_v}).
 \end{proof}
  
  \begin{cor}\label{cor:HarmCompar}
  Let $v(x)\ge 0$ be a singular $L$-harmonic function on $D$.
   The functions $v$ and $v^*$ are comparable:
there exists $c>0$ such that for all $x\in D$
 \begin{equation}\label{eq:v_and_v_star}
  c^{-1} v^*(x)\le v(x)  \le  c v^*(x).
 \end{equation}
\end{cor}
\begin{proof}
 We use the Martin representations (\ref{eq:v}), (\ref{eq:v_star}), the Corollary  \ref{cor:MartinComp}
 and the fact that $l_D(x_0,Q)>\delta>0$ for all $Q\in\partial D$.
\end{proof}
 \subsection{Perturbation formulas in the diffusion case}\label{rem:diffusionCase} 
In the  present article we exploit the perturbation formulas in the case of the 
singular operator $L=\Delta^{\alpha/2}+ b\cdot\nabla$, $1<\alpha<2$. In this short 
chapter we make a parenthesis and briefly discuss the case $\alpha=2$ and $d\ge 3$, corresponding
to the diffusion operator
$$
L=\frac12\Delta + b\cdot\nabla
$$
on $\RR^d$. The potential theory for such diffusion 
generators was studied by Cranston and Zhao\cite{CranstonZ},
and more recently by Ifra and Riahi\cite{IfraR}, Kim and Song\cite{KimSDiff} and Luks\cite{TLuksStudia}. Our methods allow
to enrich this theory by some new perturbation formulas.

We suppose that
$b\in{\mathcal K}^1_d$ and we assume additionally that $D$ is connected, i.e. it is a domain.  Recall that   Cranston and Zhao\cite{CranstonZ} 
worked under this condition and a complementary second condition  $|b|^2\in  {\mathcal K}^1_{d-1}$; 
 Kim and Song\cite{KimSDiff} suppressed the  condition on  $|b|^2$ and  considered
 signed measures in the place of $b$. 
 
 \begin{prop}\label{pr:PertGreenDiff}
  Let $L=\frac12\Delta + b\cdot\nabla$ with $ b\in  {\mathcal K}^1_d$. Then the following perturbation formula
  for the $L$-Green function $\tilde G_D$ holds
 if $x,y\in\RR^d, x\not=y.$  
\begin{equation}\label{eq:pertGDiff}
 \tilde G_D(x,y)=  G_D(x,y)+\int_D \tilde G_D(x,z) b(z)  \cdot \nabla_z  G_D(z,y) dz.
 \end{equation}
\end{prop}
\begin{proof}
Note that by \cite[Theorem 6.2]{KimSDiff}, we have the estimate
 \begin{equation}\label{eq:GreenAPriori}
   \tilde G_D(x,y)\le C|x-y|^{2-d},  \ \ x,y\in\RR^d.
 \end{equation}
 The proof of the Proposition  is  the same as the proof of \cite[Lemma 12]{KBTJ} in the case $1<\alpha<2$,
 with (\ref{eq:GreenAPriori}) replacing \cite[Lemma 7]{KBTJ}.
 \end{proof}
Let us mention that a perturbation formula for the $L$-Green function was proposed in \cite{IfraR},
 but under a restrictive assumption of boundedness of the Kato norm $\|b\|$ of $b$. 
 A simpler direct proof of the estimate  (\ref{eq:GreenAPriori}) without using
 the precise estimates \cite[Theorem 6.2]{KimSDiff} should be available.\\

 Next we obtain   a perturbation formula for the Martin kernel  of
 Laplacians with a gradient perturbation.
 \begin{prop}\label{pr:PertMartinDiff}
 Let $L=\frac12\Delta + b\cdot\nabla$ with $ b\in  {\mathcal K}^1_d$. Then the following perturbation formula
  for the $L$-Martin kernel $\tilde M_D$ holds if $x\in D$ and $Q\in\partial D$. 
  \begin{equation}\label{eq:PertMartinDiff}
\tilde M_D(x,Q) =\frac{1}{l_D(x_0,Q)}\left[ M_D(x,Q)+\int_D \tilde G_D(x,z) b(z)\cdot \nabla_z M_D(z,Q) dz\right]
\end{equation}
where $l_D(x_0,Q)$ is a continuous function on $\partial D$, equal
$$
l_D(x_0,Q)=M_D(x_0,Q)+\int_D \tilde G_D(x_0,z) b(z)\cdot \nabla_z M_D(z,Q) dz >0.
$$
 \end{prop}
 \begin{proof}
  We follow the  proof of the  Theorem  \ref{th:MartinPerturb} in the case $\alpha=2$.
 \end{proof}

 The next perturbation formula concerns the $L$-Poisson kernel $\tilde P_D(x,Q)$.
 \begin{prop}\label{pr:PertPoissonDiff}
  Let $L=\frac12\Delta + b\cdot\nabla$ with $ b\in  {\mathcal K}^1_d$. Then the following perturbation formula
  for the $L$-Poisson  kernel $\tilde P_D$ holds
 if  $x\in D$ and $Q\in\partial D$. 
\begin{equation}\label{eq:pertPoissDiff}
 \tilde P_D(x,Q)=  P_D(x,Q)+\int_D \tilde G_D(x,z) b(z)  \cdot \nabla_z  P_D(z,Q) dz.
 \end{equation}
\end{prop}
 \begin{proof}
 Observe that by the formula (\ref{eq:pertGDiff}) the function $\tilde G_D$
 has the same differentiability properties as the function $G_D$. In particular
 the inner normal derivative $\frac{\partial \tilde G_D}{\partial n}(x,Q)$ exists
 for $x\in D$ and $Q\in\partial D$.
  It is known (see \cite[page 173]{IfraR}) and  possible to prove by the Green formula
  that
  $$
  \tilde P_D(x,Q)=\frac{\partial \tilde G_D  }{\partial n}(x,Q).
  $$
  The formula (\ref{eq:pertPoissDiff}) then follows by
  differentiating of the formula (\ref{eq:pertGDiff}) in the direction of the inner normal unit vector $n$.
  We omit the technical details.
 \end{proof}

 Let us finish this section by some remarks. 
  The  formula $ \tilde P_D(x,Q)=\frac{\partial \tilde G_D  }{\partial n}(x,Q)$
   implies, like in the Laplacian case, that 
  the  $L$-Martin   and the $L$-Poisson kernels are related by the formula
\begin{equation}\label{eq:P_and_M}
 \tilde M_D(x,Q)=\frac{ \tilde P_D(x,Q)}{ \tilde P_D(x_0,Q)}.
  \end{equation}
 On the other hand, if we insert  the formula
 $
 M_D(x,Q)=\frac{ P_D(x,Q)}{  P_D(x_0,Q)}
 $
  into (\ref{eq:pertPoissDiff}), we obtain using (\ref{eq:PertMartinDiff})
  $$
  \tilde P_D(x,Q)=P_D(x_0,Q)l_D(x_0,Q)\tilde M_D(x,Q).
 $$
Evaluating the last equation at $x_0$ we obtain a  formula for the function $l_D(x_0,Q)$
intervening in the perturbation formula (\ref{eq:PertMartinDiff})
$$
l_D(x_0,Q)=\frac{\tilde P_D(x_0,Q)}{P_D(x_0,Q)}
$$
and another proof of the formula (\ref{eq:P_and_M}).
 

 \section{Relative Fatou Theorem for $L$-harmonic functions}\label{sec:boundary}
 
 We  prove in this section  an  important boundary property of $L$-harmonic functions:
 the Relative Fatou Theorem.
 As in the preceding sections, we consider a nonempty bounded $\mathcal{C}^{1,1}$ open set $D$.
 Recall the Relative Fatou Theorem in the $\alpha$-stable case. It was proved
 in \cite{MRFatou} for Lipschitz sets $D$.

\begin{thm}\label{FatouStable}
   Let $g$ and $h$ be two non-negative singular $\alpha$-harmonic functions on $D$, with Martin representations
  $$
  g(x)=\int_{\partial D} M_D(x,Q) d\mu^{(g)} (Q),\ \ \ h(x)=\int_{\partial D} M_D(x,Q) d\mu^{(h)}(Q) , \ \ x\in D.
  $$
  Then, for $\mu^{(h)}$-almost all $Q\in {\partial D}$,
  $$
  \lim_{x\ra Q} \frac{g(x)}{h(x)}=f(x)
  $$
  where $f$ is the density of the absolute continuous part of $\mu^{(g)}$ in the decomposition $\mu^{(g)} = fd\mu^{(h)} +\mu^{(g)}_{sing}$ with respect to the measure $\mu^{(h)}$,  
  and $x\ra Q$ non-tangentially.
 \end{thm}
 
 Our objective in this section is to prove an analogous limit property for
 non-negative singular  $L$-harmonic functions $u$ and $v$ on $D$.

 If we denote the integral part of the perturbation formula 
 (\ref{eq:perturb_v}) by
 $$
  I_{v^*}(x)=\int_{D}  \tG_{D}(x,w) b(w)\cdot \nabla v^*(w) dw
 $$
 then we have
 $$
 u=u^* + I_{u^*},\ \ \ 
v=v^* + I_{v^*}
 $$
 
where $u^*$ and $v^*$ are singular $\alpha$-harmonic non-negative functions.
 We write
  \begin{equation}\label{eq:2Quotients}
 \frac{u(x)}{v(x)}=\frac{u^*(x)}{v^*(x)}\displaystyle{ \frac{1+\frac{I_{u^*}(x)}{u^*(x)}}{1+\frac{I_{v^*}(x)}{v^*(x)}} }
  \end{equation}
 
 The  limit boundary behaviors of the quotients  $ \frac{u(x)}{v(x)}$ and $\frac{u^*(x)}{v^*(x)}$
 will be  related if we control the limit behavior of the quotients
 $\frac{I_{u^*}(x)}{u^*(x)}$ and  $\frac{I_{v^*}(x)}{v^*(x)}$. Thus we start with
 discussing the properties of the quotient 
 $\frac{I_{h}(x)}{h(x)}$ for a singular $\alpha$-harmonic non-negative function $h$.
 \begin{lem}\label{lem:limits_h}
 Let the Martin representation
 $h(x)=\int_{\partial D} M(x,Q) d\mu^{(h)}(Q)$ for some non-negative finite measure $\mu$ on $\partial D$. Then, if $Q\not\in {\rm supp}(\mu^{(h)})$
 $$ \lim_{x\ra Q} h(x)=0$$
 and if $Q \in {\rm supp}(\mu^{(h)})$ and $x\ra Q$ non-tangentially 
  $$ \lim_{x\ra Q} h(x)=+\infty.$$
 \end{lem}
 \begin{proof}
 The limit in the case $Q\not\in {\rm supp}(\mu^{(h)})$ follows easily from
 the Martin representation of $h$ and the Lebesgue theorem. In the case
 $Q \in {\rm supp}(\mu^{(h)})$ we use the following result of Wu \cite{Wu}.
 
 Let   $f$ be  a $\Delta$-harmonic function on $D$, corresponding via the Martin representation
 to a finite measure $\mu= \mu^{(h)}$  on    $\partial D$.
 If $Q\in {\rm supp}\mu$, then
 $$
 \liminf_{x\ra Q} f(x)>0,
 $$
provided $x\ra Q$ non-tangentially.
We have, on  $  D$ of class ${\mathcal C}^{1,1}$

\begin{eqnarray*}
&& f(x)=\int_{\partial D} P^\Delta_D(x,y)\mu(dy)\le c\int_{\partial D} \frac{\delta_D(x)}{|x-y|^d}\mu(dy)\\
&& =c \delta_D(x)^{1-\frac{\alpha}{2} }\int_{\partial D} \frac{\delta_D(x)^{\frac{\alpha}{2} }}{|x-y|^d}\mu(dy)
 \le C  \delta_D(x)^{1-\frac{\alpha}{2} } \int_{\partial D} M_D(x,y) \mu(dy)
 \end{eqnarray*}
 Consequently
 $$
 h(x)\ge \frac{1}{C} \frac{f(x)}{ \delta_D(x)^{1-\frac{\alpha}{2} }}
 $$
 and the second part of the Lemma follows. 

 \end{proof}
  \begin{lem}\label{lem:quotient_bounded}
  The quotient $\frac{I_{h}(x)}{h(x)}$ is bounded. More exactly, there exists $c>0$
 such that
 \begin{equation}\label{eq:Quotient_bounded}
  c^{-1} \le 1+\frac{I_{h}(x)}{h(x)} \le c.
 \end{equation}
   \end{lem}
 \begin{proof}
 Observe that by Corollary \ref{cor:HarmCompar} and the formula (\ref{eq:perturb_v}),
 the quotient $\frac{I_{v^*}(x)}{v^*(x)}$ is bounded. More exactly, there exists $c>0$
 such that
 \begin{equation*}
  c^{-1} \le 1+\frac{I_{v^*}(x)}{v^*(x)} \le c.
 \end{equation*}
 As the function $l_D(x_0,Q)$ is bounded, any
 singular $\alpha$-harmonic non-negative function $h$ is of the form $v^*$
 for a  singular $L$-harmonic non-negative function $v$. 
  \end{proof}
 
 By (\ref{eq:GreenCompar}), if we denote
 $$
 J_h(x)= \int_{D}  G_{D}(x,w) b(w)\cdot \nabla h(w) dw
 $$
 then
 $$
 I_h(x)\sim J_h(x)
 $$
 In particular, by Lemma \ref{lem:quotient_bounded}, the quotient $J_h(x)/h(x)$ is bounded.
 We prove a much stronger property of this quotient in the following lemma.
 \begin{lem}\label{lem_Quotient_lim0}
 Let $h$ be a non-negative singular $\alpha$-harmonic function on $D$, with the Martin representation
 $h(x)=\int_{\partial D}  M_D(x,Q)d\mu^{(h)}(Q)$ for a finite measure $\mu^{(h)}$ on $\partial D$. 
 Then, when $Q\in {\rm supp}(\mu^{(h)})$ and $x\ra Q$ non-tangentially, we have
 $$
   \lim_{x\ra Q} \frac{J_h(x)}{h(x)}=0.
 $$
\end{lem}
 \begin{proof}
We will show that $\frac{G_{D}(x,w) h(w)}{h(x) \delta_D(w)}$ is uniformly integrable in $x \in D$ against the measure $|b(w)| dw$. Let $\varepsilon>0$. Since $J_h(x)/h(x)$ is bounded it suffices to show that there is $\delta>0$ such that 
\begin{equation}
\int_F \frac{G_{D}(x,w) h(w)}{h(x) \delta_D(w)} |b(w)| dw \le \varepsilon,\label{eq1:J_h}
\end{equation}
provided $\lambda(F) < \delta$. Here, $\lambda$ denotes the Lebesgue measure on $\RR^d$. 
First, we note that
\begin{align}
& \int_{F} \frac{G_{D}(x,w) h(w)}{h(x) \delta_D(w)} |b(w)| dw  \notag\\
& =\int_{F} \int_{\partial D} \frac{G_{D}(x,w) M_D(w,Q)}{h(x) \delta_D(w)}  d\mu^{(h)}(Q) |b(w)| dw \notag\\
& = \int_{\partial D} \frac{M_D(x,Q)}{h(x)}\left(\int_{F}  \frac{G_{D}(x,w) M_D(w,Q)}{M_D(x,Q) \delta_D(w)}  |b(w)| dw\right)
 d\mu^{(h)}(Q). \label{eq2:J_h}
\end{align}
The function $\frac{G_{D}(x,w) G_D(w,y)}{G_D(x,y) \delta(w)}$ is uniformly integrable in $x,y \in D$ against $|b(w)|dw$ (see the proof of \cite[Lemma 11]{KBTJ}). Hence, there exists $\delta>0$ such that for $\lambda(F)<\delta$,
$$
\int_F \frac{G_{D}(x,w) G_D(w,y)}{G_D(x,y)\delta_D(w)} |b(w)|dw< \varepsilon, \qquad x,y \in D,
$$
and consequently
\begin{align*}
&\int_F \frac{G_{D}(x,w) M_D(w,Q)}{M_D(x,Q)\delta_D(w)} |b(w)|dw = \int_F \lim_{D \ni y\to Q} \frac{G_{D}(x,w) G_D(w,y)}{G_D(x,y)\delta_D(w)} |b(w)|dw\\
& = \lim_{D \ni y\to Q} \int_F  \frac{G_{D}(x,w) G_D(w,y)}{G_D(x,y)\delta_D(w)} |b(w)|dw \le \varepsilon.
\end{align*}
Now, (\ref{eq1:J_h}) follows from (\ref{eq2:J_h}) and Martin representation of $h$.
For $Q \in {\rm supp} \mu^{(h)}$, $\lim_{D \ni x \to Q} h(x) = \infty$ from the Lemma \ref{lem:limits_h}. Hence, by uniform integrability,
\begin{align*}
\lim_{D \ni x\ra Q} \frac{|J_h(x)|}{h(x)} & \le  c \lim_{D \ni x\ra Q} \int_{D} \frac{G_{D}(x,w) h(w)}{h(x) \delta_D(w)} |b(w)| dw \\
& =  c\int_{D} \lim_{D \ni x\ra Q} \frac{G_{D}(x,w) h(w)}{h(x) \delta_D(w)} |b(w)| dw = 0.
\end{align*}
 \end{proof}
Now, we return to the Relative Fatou Theorem for $L$-harmonic functions.
 Let $u$ and $v$ be two non-negative singular  $L$-harmonic functions on $D$. By Theorem \ref{th:MartinRepr}, they
 have a Martin representation
 $$
  u(x)=\int_{\partial D} \tilde M_D(x,Q)d\mu(Q),\ \ 
 v(x)=\int_{\partial D} \tilde M_D(x,Q)d\nu(Q),\ \ \ x\in D
 $$
 where $\mu$ and $\nu$ are two Borel finite measures concentrated on $\partial D$.
 
We decompose the measure $\mu$ into its absolutely continuous and singular parts
 with respect to the measure $\nu$
 $$
 d\mu= f\;d\nu+d\mu_{sing}
 $$
with a non-negative function $f\in L^1(\nu)$ and $\nu({\rm supp}(\mu_{sing}))=0$.

\begin{thm}\label{th:Fatou}
 ({\bf Relative Fatou Theorem})
 For $\nu$-almost every point $Q\in\partial D$ we have
 \begin{equation}\label{eq:lim_Fatou}
  \lim_{x\ra Q}\frac{u(x)}{v(x)}=f(Q)
 \end{equation}
when $x\ra Q$ non-tangentially.
\end{thm}
\begin{proof}
 We will use the Relative Fatou Theorem for the singular  $\alpha$-harmonic functions $u^*$
 and $v^*$ defined  according to (\ref{eq:v_star}).
 
 Let  $Q\in {\rm supp}(\nu)\setminus {\rm supp}(\mu)$. Then, if $x\ra Q$, 
 $v^*(x)\ra \infty$ and  $u^*(x)\ra 0$, so $\lim_{x\ra Q}\frac{u^*(x)}{v^*(x)} =0$.
  The  formulas (\ref{eq:Quotient_bounded}) and (\ref{eq:2Quotients}) imply that in this case 
  $$\lim_{x\ra Q}\frac{u(x)}{v(x)}=0.$$
 
 Let us consider the case $Q\in {\rm supp}(\nu)\cap {\rm supp}(\mu)$.  
 As 
 $$
\frac{d\mu(Q)}{l_D(x_0,Q)}= f\; \frac{d\nu(Q)}{l_D(x_0,Q)}+\frac{d\mu_{sing}(Q)}{l_D(x_0,Q)},
 $$
 the Relative Fatou Theorem for the singular  $\alpha$-harmonic functions $u^*$
 and $v^*$ says that  for $\nu$-almost every point $Q\in\partial D$ 
 $$
  \lim_{x\ra Q}\frac{u^*(x)}{v^*(x)}=f(Q)
 $$
when $x\ra Q$ non-tangentially.
The formula (\ref{eq:lim_Fatou}) then follows by the formula
 (\ref{eq:2Quotients}) and the Lemma \ref{lem_Quotient_lim0}.
\end{proof}

\section*{Acknowledgements}
We thank  Krzysztof Bogdan, Jacek Ma\l ecki and Micha\l\ Ryznar  for discussions about this paper.
\bibliographystyle{abbrv}

\end{document}